\documentclass[10pt]{amsart}
\usepackage{amssymb}

 \newtheorem{thm}{Theorem}[section]
 \newtheorem{cor}[thm]{Corollary}
 
 \newtheorem{lem}[thm]{Lemma}
 
 \theoremstyle{definition}
 \newtheorem{defn}[thm]{Definition}
 \newtheorem{nta}[thm]{Notation}
 \theoremstyle{remark}

\newcommand{\ord}{{\mathrm{o}}}
\newcommand{\pg}{{\mathcal{P}}}

\begin{document}

\title
{Edge-maximality of power graphs of finite cyclic groups}

\author{Brian Curtin}
\author{G.~R.~Pourgholi}

\address{{\bf Brian Curtin}: Department of Mathematics and Statistics, University of South Florida, Tampa FL, 33620}
\email{bcurtin@usf.edu}

\address{{\bf G. R. Pourgholi}: School of Mathematics, Statistics and Computer Science,
University of Tehran, Tehran 14155-6455, I.~R.~Iran}

\email{pourgholi@ut.ac.ir; pourgholi@alumni.ut.ac.ir;
gh.reza.pourgholi@gmail.com}

\begin{abstract}
We show that among all finite groups of any given order, the
cyclic group of that order has the maximum number of edges in its
power graph.

\vskip 3mm

\end{abstract}

\maketitle

\begin{center}
{\bf
The final publication is available at link.springer.com\\
DOI: 10.1007/s10801-013-0490-5\\
Several minor corrections to the published version appear in the ArXiv version}
\end{center}
\bigskip

\section{Introduction}

In this paper, we resolve  in the affirmative a conjecture of
Mirzargar et al. \cite[Conjecture 2]{m}  concerning the number of
edges  in the power graph of a finite group.   Motivated by the
work of Kelarev and Quinn \cite{KQ1,KQ2,k,KQS}, Chakrabarty,
Ghosh, and Sen \cite{ch} introduced undirected power graphs to
study semigroups and groups. Other relevant work includes
\cite{c1,c2}.  The reader is encouraged to see \cite{AbKeCh} which
surveys the literature to date.

\begin{defn}
Let $G$ be a finite group.  Let $\langle g \rangle$
denote the cyclic subgroup of $G$ generated by $g\in G$.
\begin{enumerate}
\item The {\em directed power graph} $\overrightarrow{\pg}(G)$ of  $G$ is the directed graph with vertex set $G$ and  directed edge set
  $\overrightarrow{E}(G) = \{(g,h)\,|\, g, h\in G,\,
                                                       h\in\langle g \rangle-\{g\}\}$.
That is, there is an edge from one group element to a second whenever the second  is a positive power of the first and distinct from the first.

\item
\label{defn:first}
The {\it undirected  power graph} (or {\em power graph} ) ${\pg}(G)$ of $G$ is the undirected graph with vertex set $G$ and edge set
$E(G)=\{ \{g,h\}   \, | \, (g,h)\in \overrightarrow{E}(G) \hbox{ or }
                                 (h,g)\in\overrightarrow{E}(G) \}$.
That is,  two distinct group elements  are adjacent whenever one
of them is a positive power of the other.
\end{enumerate}
\end{defn}

We recall the following property of directed power graphs of cyclic groups.

 \begin{thm} \cite[Main Theorem]{a}
\label{cor:maximume} Among all finite groups of a given order, the
cyclic group of that order has the maximum number of edges in its
directed power graph.
\end{thm}

Our main theorem, which resolves \cite[Conjecture 2]{m},  is the undirected analog of Theorem \ref{cor:maximume}.

\begin{thm}
\label{thm:main} Among all finite groups of a given order, the
cyclic group of that order has the maximum number of edges in its
power graph.
\end{thm}

Theorem \ref{thm:main} gives that (ii) implies (iii) in the following.  The others are trivial.

\begin{cor}
\label{cor:isomisom}
Let $G$ be a finite group.  Then the following are equivalent for all $n\geq 1$.\begin{enumerate}
\item $\pg(G)\cong\pg(\mathbb{Z}_n)$.
\item $|E(G)|=|E(\mathbb{Z}_n)|$.
\item $G\cong \mathbb{Z}_n$.
\end{enumerate}
\end{cor}

 One special case of Theorem \ref{thm:main} is already known.

 \begin{thm}\cite[Theorem 2.12]{ch}
 \label{lem:PGcyclicprimepower1}
 A finite group has a complete power graph
if and only if it is cyclic and has prime power order.
\end{thm}

\section{Edges in power graphs}

Let $G$ be a finite group.  For $g \in G$, let
$\ord(g)$ denote the order of $g$ as a group element and let $\deg(g)$ denote the degree of $g$ as a vertex of $\pg(G)$.   Throughout $\phi(n)$ shall
denote the Euler totient function of the natural number $n$.

Pick $g\in G$.  Observe that $g$ has out-degree $\ord(g) -1$ since there is a directed edge from $g\in G$ to each element of $\langle g\rangle-\{g\}$.  There is a directed edge from each $h\in G-\{g\}$ to $g$ for which $g\in\langle h \rangle$, so the in-degree of $g$ is
$|\{ h\in  G-\{g\}\,|\, g\in\langle h\rangle\}|$. To account for directed edges which give the same undirected edge in the power graph of $G$, we introduce the following set.

\begin{defn}
The set of {\em bidirectional edges} $\overleftrightarrow{E}(G)$
of $\overrightarrow{\pg}(G)$  is set of unordered pairs $\{
\{g,h\} \,|\, (g,h)\in \overrightarrow{E}(G) \hbox{ and }
                        (h,g)\in\overrightarrow{E}(G) \}$.
That is, $\overleftrightarrow{E}(G)$ consists of pairs of distinct elements, each of which is a positive power of the other.
\end{defn}

\begin{lem}
\label{lem:bidi=gensame}
Let $G$ be a finite group, and let $g$, $h$ be distinct elements of $G$.  Then
   $\{g,h\}\in\overleftrightarrow{E}(G)$
if and only if
   $\langle g\rangle=\langle h\rangle$.
\end{lem}

\begin{proof}
Straightforward from the definition of adjacency in the directed power graph.
\end{proof}

\begin{lem}\cite[Theorem 4.2]{ch}
\label{lem:edgesizes}
Let $G$ be a finite group of order $n$. Then
\begin{eqnarray}
\label{eq:dEGsize}
    |\overrightarrow{E}(G)|
             &=&  \sum_{g\in G} (\ord(g) -1), \\
\label{eq:bdEGsize}
     |\overleftrightarrow{E}(G)|
             &=& \frac{1}{2} \sum_{g\in G}( \phi(\ord(g))-1 ), \\
\label{eq:EGsize}
     |E(G)|&=& |\overrightarrow{E}(G)|-|\overleftrightarrow{E}(G)|
               = \frac{1}{2}\sum_{g\in G} \left(2\ord(g) - \phi(\ord(g)) - 1\right).
\end{eqnarray}
\end{lem}

\begin{proof}
The sum in (\ref{eq:dEGsize}) adds out-degrees of vertices, and
thus counts each directed edge once. Now (\ref{eq:bdEGsize}) follows
from Lemma \ref{lem:bidi=gensame} and the fact that a cyclic group of order $\ord(g)$ has ($\phi(\ord(g))$)-many generators.  Indeed, $\phi(\ord(g))-1$ such edges leave $g$, and summing over all $G$ double counts these edges.  For the first equality
in (\ref{eq:EGsize}), count the edges in the directed power graph,
and subtract one for each pair of oppositely oriented directed edges to avoid double counting.
The second equality in (\ref{eq:EGsize}) follows from (\ref{eq:dEGsize})  and (\ref{eq:bdEGsize}).
\end{proof}

We give the number of edges in the undirected power graph of
$\mathbb{Z}_n$.  We use the following notation.

\begin{nta}
\label{nta:nfactored}
Let $n$ be a positive integer.
Write
    $n = p_1^{\alpha_1}p_2^{\alpha_2}\cdots p_k^{\alpha_k}$
for primes
     $p_1 < p_2 < \cdots <p_k$.
Let $q=p_1$ and $p=p_k$ be the least and greatest prime divisors of $n$, and abbreviate $\beta=\alpha_1$ and $\alpha=\alpha_k$.
\end{nta}

It is well-known (see  \cite[p.~27]{a'}, for instance) that
\begin{equation}
\label{eq:phi(n)primes}
 \phi(n) = p_1^{\alpha_1-1}(p_1-1)p_2^{\alpha_2-1}(p_2-1)
                                     \cdots p_k^{\alpha_k-1}(p_k-1).
\end{equation}
As a consequence, we have \cite[p.~28]{a'}
\begin{equation}
\label{eq:phi(product)}
\phi(nm) =\phi(n)\phi(m) \frac{\gcd(n,m)}{\phi(\gcd(n,m))}.
\end{equation}

\begin{lem}
\label{lem:sumZn} (See also \cite{burton:ENT}, page 143, exercise 5)
With  Notation \ref{nta:nfactored},
\begin{eqnarray}
\label{eq:sumordZn}
 \sum_{z\in\mathbb{Z}_n} \ord(z)
   &=&\sum_{d|n} \phi(d)d
  \ = \      \prod_{h=1}^{k} \frac{p_h^{2\alpha_h+1}+1}{p_h+1},\\
\label{eq:sumphiordZn}
  \sum_{z\in \mathbb{Z}_n} \phi(\ord(z))
  &=& \sum_{d|n} \phi(d)^2
  \ = \  \prod_{h=1}^{k}  \frac{p_h^{2\alpha_h}(p_h-1)+2}{p_h+1}.
\end{eqnarray}
\end{lem}

\begin{proof}
For each $z\in \mathbb{Z}_n$, $\ord(z)$ is a divisor $d$ of $n$.  For each divisor $d$, there are $\phi(d)$-many other elements of $\mathbb{Z}_n$ with the same order.  Thus the first equality in (\ref{eq:sumordZn}) holds.
Similarly, for each of the $\phi(d)$-many elements of $\mathbb{Z}_n$
with the same order as $z$, $\phi(\ord(z)) = \phi(d)$. Thus the first equality in (\ref{eq:sumphiordZn}) holds.

When $n=1$, its only divisor is 1, and $\phi(1)=\phi(1)^2=1$.  Thus both
sums on the left are 1. There are no prime divisors of 1, so the
product on the right is empty, and hence 1. Thus both second
equalities holds when $k=0$. Assume that $n$ has $k\geq 1$ distinct prime divisors and that both second equalities
holds for all $n$ with at most $k-1$ distinct prime divisors.

Partition the divisors of $n$ according to the highest power of
$q$ which divides it. As $d$ runs over
         $d|n$ with $q^\ell|f$ and
         $q^{\ell+1}\not|f$, we have
           $d = f q^\ell$
as $f$ runs over the divisors of $n/q^{\beta}$. Since
        $q\not| n/q^{\beta}$,
(\ref{eq:phi(n)primes}) gives that
      $\phi(f q^\ell)= q^{\ell-1}(q-1)\phi(f)$.
Compute
\[
 \sum_{d|n} \phi(d)d
     =\sum_{\ell=0}^\beta \sum_{f|n/q^\beta} \phi(q^\ell f)f q^\ell
      =    \sum_{\ell=0}^\beta \phi(q^\ell)q^\ell \sum_{f|n/q^\beta} \phi(f)f
      = \left(1+\sum_{\ell=1}^\beta q^{\ell-1}(q-1)q^\ell\right)
           \sum_{f|n/q^\beta} \phi(f)f,
 \]
and
\[
\sum_{\ell=1}^\beta q^{2\ell-1}(q-1)
   =  \sum_{\ell=0}^{\beta-1} (q-1)q^{2\ell+1}
     =  \frac{q-1}{q} \cdot \sum_{\ell=0}^{\beta-1}(q^2)^\ell
      =    q(q-1) \cdot \frac{(q^{2\beta}-1)}{q^2-1}
   =  q\cdot\frac{q^{2\beta}-1}{q+1}
       =  \frac{q^{2\beta+1}+1}{q+1} - 1.
\]
Now the second equality in (\ref{eq:sumordZn}) follows by induction.

Similarly,
\[
 \sum_{d|n} \phi(d)^2
   = \sum_{\ell=0}^\beta \sum_{f|n/q^\beta} \phi(q^\ell f)^2
     =    \sum_{\ell=0}^\beta \phi(q^\ell)^2 \sum_{f|n/q^\beta} \phi(f)^2
   = \left(1+\sum_{\ell=1}^\beta (q^{\ell-1}(q-1))^2\right)
            \sum_{f|n/q^\beta} \phi(f)^2,
\]
and
\[
 \sum_{\ell=1}^\beta (q^{\ell-1}(q-1))^2
       = (q-1)^2 \sum_{\ell=0}^{\beta-1} q^{2\ell}
       = (q-1)^2 \frac{q^{2\beta}-1}{q^2-1}
       =  \frac{(q-1)(q^{2\beta}-1)}{q+1}.
\]
Now the second equality in (\ref{eq:sumphiordZn})  follows by induction.
\end{proof}

\begin{cor}
\label{cor:peelsumphiordZn}
With Notation \ref{nta:nfactored},\footnote{This corollary was after the published version.  It makes explicit a fact proved in Lemma \ref{lem:sumZn} for later use.}
\[
 \sum_{d|n} \phi(d)d = \frac{q^{2\beta+1}+1}{q+1}\cdot\sum_{f|n/q^\beta} \phi(f)f. 
\]
\end{cor}

\begin{cor}
With Notation \ref{nta:nfactored},
\[
 |E(\mathbb{Z}_n)|
     = \sum_{d|n} \phi(d)(d-\frac{\phi(d)}{2})  -\frac{n}{2}
        =  \prod_{h=1}^{k} \frac{p_h^{2\alpha_h+1}+1}{p_h+1}
       - \frac{1}{2} \prod_{h=1}^{k}  \frac{p_h^{2\alpha_h}(p_h-1)+2}{p_h+1}
       - \frac{n}{2}.
\]
\end{cor}

\begin{lem}
With  Notation \ref{nta:nfactored}, pick $z \in\mathbb{Z}_{n}$, and write $e=\ord(z)$. Then
\begin{equation}
\label{eq:degz-Zn}
\deg(z) = e -1 -\phi(e) + \sum_{{ d | n/e}} \phi(de)
            = e -\phi(e) -1
                  +\phi(e) \sum_{d | n/e} \phi(d)\frac{\gcd(d,e)}{\phi(\gcd(d,e))}.
\end{equation}
\end{lem}

\begin{proof}
The term $\ord(z) -1=e-1$ is the out-degree of $z$. There is a directed edge from each element
         $x\in \mathbb{Z}_{n}$ to $z$
whenever
         $\ord(z)|\ord(x)$.
For such an $x$, $\ord(x)$ is $\ord(z)$ times a divisor of  $n/\ord(z)$.
There are $\phi(k\ord(z))$-many generators of $\langle x\rangle$.  Thus the in-degree of $z$ is $\sum_{d | n/e} \phi(d e)$.
However, to avoid double counting when dropping the orientation, we must exclude those elements with the same order as $z$, i.e., the case $k=1$ where $\phi(ke)=\phi(e)$.       Hence the first equality holds.  The second equality follows from (\ref{eq:phi(product)}) since the summand for $d=1$  is 1.
\end{proof}

\begin{lem}
\label{lem:Inducedsubgaph} Let $G$ be a group, and let $H$ be a
subgroup of $G$.
\begin{enumerate}
\item  $\overrightarrow{\pg}(H)$ is an induced subgraph of $\overrightarrow{\pg}(G)$.
\item All out-edges from an element of $H$ terminate at an element of $H$.
\item $\pg({H})$ is an induced subgraph of $\pg(G)$.
\end{enumerate}
In particular, the adjacencies and non-adjacencies between
elements of $H$ are the same in $\overrightarrow{\pg}(H)$ and
$\overrightarrow{\pg}(G)$, and similarly for $\pg({H})$ and
$\pg(G)$.
\end{lem}

\begin{proof}
Straightforward from the definitions.
\end{proof}

\section{Some inequalities}\label{sec:prelim}

We develop some inequalities.

\begin{lem}
\label{lem:inequalities} Let $G$ be a finite group.  Then the following hold.
\begin{enumerate}
\item  $|G|\leq |E(G)|+1$
with equality if and only if $G$ is an elementary abelian
$2$-group.

\item $|E(G)|\leq  |\overrightarrow{E}(G)|$ with equality if and only if
          $G$ is an elementary abelian $2$-group.
\item
\label{item:equality}  
\footnote{This statement of Lemma \ref{item:equality} went awry in the published version.  The  proof is essentially the same.}
$2 |\overleftrightarrow{E}(G)|\leq|\overrightarrow{E}(G)|-|G|+1$ 
with equality if and only if 
every nonidentity element of $G$ has prime order, i.e., $G$ is an EPO-group. 
\end{enumerate}
\end{lem}

\begin{proof}
(i): Every nonidentity element of $G$ is adjacent to the identity,
so
           $|G|\leq |E(G)|+1$.
Equality holds
     if and only if
every nonidentity element of $G$ has order two.

(ii):  Every undirected edge arises from a directed edge, so the
inequality holds.  By (\ref{eq:EGsize}), $|E(G)| =
|\overrightarrow{E}(G)|$
  if and only if
there are no bidirectional edges
    if and only if
for all $g\in G$, $g$ is the only generator of   $\langle
g\rangle$
    if and only if
for all $g\in G$, one is the only number both less than and
coprime to $\ord(g)$
    if and only if
every element of $G$ has order two.

(iii):  The ($|G|-1$)-many edges to the identity are not
bidirectional, and the bidirectional edges come from pairs of
directed edges. Thus
     $2 |\overleftrightarrow{E}(G)|\leq |\overrightarrow{E}(G)|-|G|+1$.
Now equality holds in (iii)
     if and only if
 every edge not incident to the identity is bidirectional.     
If some $g\in G$ has composite order, say $\ord(g)=pm$ for a prime
$p$ and $m>1$, then the edge between $g$ and $g^p$ is not
bidirectional.  Thus equality fails in this case.   If every element of $G$ has prime order, then each element generate a cyclic group of prime order.  These subgroups only have the identity in common.  Thus every edge not incident to the identity is bidirectional, so equality holds. 
\end{proof}

The remaining inequalities in this section pertain to $\mathbb{Z}_n$, and so are number theoretic in nature.

 \begin{lem}\cite[Main Theorem]{a}
 \label{lem:PGcyclicprimepower2}
Let $G$ be a group of finite order $n$.
Then
$\sum_{g\in G} \ord(g) \leq \sum_{z\in\mathbb{Z}_n} \ord(z)$,
with equality if and only if  $G$ is isomorphic to $\mathbb{Z}_n$.
\end{lem}

Theorem \ref{cor:maximume} is a straightforward consequence of
(\ref{eq:dEGsize}) and Lemma \ref{lem:PGcyclicprimepower2}.
Our next goal is to improve the bound
        ${\sum_{d|n} \phi(d)d > {n^2}/{p}}$
given in  \cite[Lemma D]{a}.

\begin{thm}
\label{thm:phiddineq}
 \footnote{The published version of Theorem \ref{thm:phiddineq}, did not note the forbidden cases.  This was our mistake. The corrected proof is a careful reworking of the original. Only (\ref{eq:betterbound}) was required in the sequel, so this error had no impact upon the sequel.}
With Notation \ref{nta:nfactored}, if $n>1$ does not have one of the following  prime factorizations
       $2^a$,
       $2^a 3^b$,
       $2^a 3^b 5^c$, or
       $2^a 3^b 5^c7^d$ for positive $a$, $b$, $c$, $d$,
 then
\begin{equation}
\label{eq:betterboundodd}
\sum_{d|n} \phi(d)d
   \geq \left( \prod_{h=1}^{k} \frac{p_h+1}{p_h}\right)\cdot \frac{n^2}{p}.
\end{equation}
If $n$ is not a power of 2, then
\begin{equation}
\label{eq:betterbound}
\sum_{d|n} \phi(d)d  \geq \left( \frac{q+1}{q}\right)\cdot \frac{n^2}{p}.
\end{equation}
\end{thm}

\begin{proof}
We prove (\ref{eq:betterboundodd}) by induction on the number $k$ of distinct prime divisors of $n$.
We launch the induction by showing that (\ref{eq:betterboundodd}) holds for the  ``first'' sets of prime factors that are not forbidden.
Abbreviate
\[ S(n) = \sum_{d|n} \phi(d)d, \quad 
   R(n) =  \left( \prod_{h=1}^{k} \frac{p_h+1}{p_h}\right)\cdot \frac{n^2}{p}; \quad
\hbox{ so}\qquad  
    \frac{S(n)}{R(n)} =     p\cdot  
    \prod_{h=1}^{k}\frac{\displaystyle{p_h^2+\frac{1}{p_h^{2\alpha_h-1}}}}
                                   {(p_h+1)^2}.
   \]
Note that (\ref{eq:betterboundodd}) holds when $S(n)/R(n)\geq 1$.  The term $1/p_h^{2\alpha_h-1}$ is a decreasing function of $\alpha_h$ and $\alpha_h$ can be arbitrarily large.  Thus to verify (\ref{eq:betterboundodd}), it suffices to show that 
\[ T(n):= p\cdot  
    \prod_{h=1}^{k}\frac{\displaystyle{p_h^2}}
                                   {(p_h+1)^2} \geq 1.
\]
Note that $T(n)$ does not depend upon the exponents.     We readily compute that
$T(3^{\alpha_2}) =  27/16$, 
$T(2^{\alpha_1} 5^{\alpha_2})=5/4$, 
$T(2^{\alpha_1}3^{\alpha_2}7^{\alpha_3})=343/256$,
$T(2^{\alpha_1}3^{\alpha_2}5^{\alpha_3}11^{\alpha_4})=33275 / 20736$,
$T(2^{\alpha_1}3^{\alpha_2}5^{\alpha_3}7^{\alpha_4} 11^{\alpha_5})=1630475/1327104$.  This establishes the initial step for the induction on $k$ for $n$ not of a forbidden form. 

We now proceed to the inductive step. Suppose 
   $n=p_1^{\alpha_1} p_2^{\alpha_2} \cdots p_k^{\alpha_k}$
with $k\geq 2$, and let 
   $n'=p_1^{\alpha_1} p_2^{\alpha_2} \cdots p_{k-1}^{\alpha_{k-1}}$.
 Assume $n'$ is not one of the forbidden forms, so  $p_k\geq p_{k-1}+2$
 and $S(n')\geq R(n')$ by induction.
 Observe that by Corollary \ref{cor:peelsumphiordZn}
\[ S(n) = S(n')\cdot \frac{p_k^{2\alpha_{k}+1}+1}{p_k + 1}, \qquad
   R(n) = R(n') \cdot p_{k-1}\cdot \frac{p_k+1}{p_k}\cdot \frac{p_k^{2\alpha_k}}{p_k}.\]
Since $p_{k-1}\leq p_k-2$, we compute   
\[\frac{S(n)}{R(n)} = \frac{S(n')}{R(n')}\left(\frac{p^2(p+p^{-2\alpha_k})}{(p+1)^2p_{k-1}}\right) \geq
             \frac{S(n')}{R(n')}\left(\frac{p^3+p^{-2\alpha_k+2}}{p^3-3p-2 }\right) \geq 1.
 \]            
Thus (\ref{eq:betterboundodd}) holds by induction.

Except in the forbidden cases this implies (\ref{eq:betterbound}) since each factor $(p_h+1)/p_h>1$. Thus, since $n$ is not a power of 2, we need only consider three cases for  (\ref{eq:betterbound}). 
Abbreviate
\[ Q(n) = \frac{q+1}{q}\frac{n^2}{p}, \quad  U(n) =  \frac{p q}{q+1}\prod_{h=1}^{k}\frac{\displaystyle{p_h}}
                                   {p_h+1}; \quad\hbox{ so }\quad 
   \frac{S(n)}{Q(n)} = \frac{p q}{q+1}\prod_{h=1}^{k}\frac{\displaystyle{p_h+\frac{1}{p_h^{2\alpha_h}}}}
                                   {p_h+1} \geq U(n).
\]
To prove (\ref{eq:betterbound}) , it suffice to show that $U(n)\geq 1$ for forbidden $n$ which are not powers of 2.
One computes
$U(2^{\alpha_1} 3^{\alpha_2})=1$, 
$U(2^{\alpha_1}3^{\alpha_2}5^{\alpha_3})=25/18$,
$U(2^{\alpha_1}3^{\alpha_2}5^{\alpha_3}7^{\alpha_4})=245/144$.
\end{proof}

A hybrid of $T(n)$ and $S(n)/R(n)$ allows one to verify (\ref{eq:betterboundodd}) for families with arbitrary powers, yielding
$\frac{S}{R}(2) = 1$, 
$\frac{S}{R}(2\cdot 3 \cdot 5^{\alpha_3})  \geq  875/864$,
 $\frac{S}{R}(2^ 2 \cdot3\cdot  5\cdot 7^{\alpha_4})  \geq 184877/184320$,
$\frac{S}{R}(2\cdot 3^{\alpha_2} \cdot 5^{\alpha_3}\cdot 7^{\alpha_4})     \geq 1225/1152$.
Thus (\ref{eq:betterboundodd}) holds for these arguments for all positive integers $\alpha_2$, $\alpha_3$, $\alpha_4$. Since $S(n)/R(n)$ is decreasing in the $\alpha_h$,   (\ref{eq:betterboundodd}) can be shown to fail for some $n$ by showing $S(n')/R(n')<1$ for some  $n'$ with the same distinct prime divisors but possibly smaller, positive exponents.   We compute 
$\frac{S}{R}(2^2) = 11/12$, 
$\frac{S}{R}(2\cdot 3) =  7/8$,
$\frac{S}{R}(2^2\cdot 3 \cdot 5) = 539/576$,
$\frac{S}{R}( 2\cdot 3^2 \cdot 5) = 427/432$,
$\frac{S}{R}(2^3\cdot 3 \cdot 5\cdot 7) = 90601/92160$,
$\frac{S}{R}(2^ 2 \cdot3^2\cdot  5\cdot 7) = 201971/207360$,
$\frac{S}{R}(2^ 2 \cdot3\cdot  5^2\cdot 7) = 1725031/1728000$.
Thus (\ref{eq:betterboundodd}) fails for numbers with forbidden prime factorizations other than those noted above.
%

\begin{lem}
With  Notation \ref{nta:nfactored}, pick $z \in\mathbb{Z}_{n}$, and write $e=\ord(z)$. Then
\[
\deg(z) \geq
           e  -1    +\phi(e) (\frac{n}{e}-1),
\]
with equality if and only if $e$ and $n/e$ are coprime.
\end{lem}

\begin{proof}
Consider (\ref{eq:degz-Zn}). The term
        ${\gcd(d,e)}/{\phi(\gcd(d,e))}$
is at least 1 with equality if and only if  $\gcd(d,e)=1$.  If $n/e$ has any divisors not coprime to $e$, then the inequality must be strict.
\end{proof}

\begin{lem}
With Notation \ref{nta:nfactored},
\begin{equation}
\label{eq:phigeqn/p}
 {\phi(n)} \geq  \frac{n}{p},
\end{equation}
with equality if and only if $n=2^\alpha 3^\beta$
   with $\alpha \geq 1$ and $\beta\geq 0$.
\end{lem}

\begin{proof}
If $n=p^\alpha$, then
$\phi(p^\alpha)=p^{\alpha-1}(p-1) \geq p^{\alpha-1} =  p^\alpha/p$.
The inequality holds in this case, with equality if and only if $p=2$.  Now assume that $n$ has $k\geq 2$ distinct prime factors.
By  (\ref{eq:phi(n)primes}),
    $\phi(n)/n = (p_1 -1) (p_2-1) \cdots (p_k-1)/(p_1 p_2 \cdots p_k)$.
Note that $p_{i+1} - 1\geq p_{i-1}$ $(1\leq i \leq k-1)$, with equality precisely when $k=2$, $p_1=2$, and $p_{2}=3$.  Telescoping the middle terms gives $\phi(n)/n\geq (p_1-1)/p_k = (q-1)/p\geq 1/p$, with equality under the stated conditions, i.e., $n=2^\alpha 3^\beta$ for positive $\alpha$ and nonnegative $\beta$.
\end{proof}

\begin{lem}
\label{lem:avez}
 With Notation \ref{nta:nfactored}, assume that $n$ is not equal to
$2^\alpha$ with  $\alpha\geq 1$.
Then \footnote{Fixed a sign error in first equation--only used in Lemma \ref{lem:lemma1}, but doesn't change that result.}
\begin{align*}
\sum_{z\in \mathbb{Z}_n} \deg(z)
         &\geq    2 \frac{n^2}{p} - \frac{n}{p}-1&
                    \hbox{if $\phi(n) \leq n/q$,} \\
\sum_{z\in \mathbb{Z}_n} \deg(z)
            &>(n-1)\left(\frac{n}{q}+1\right) &
                    \hbox{if $\phi(n)> n/q$.}
\end{align*}
\end{lem}

\begin{proof}
For any   prime $p$ (other than 2), both
      $\phi(p^\alpha)=p^{\alpha-1}(p-1)>p^\alpha/p$
by (\ref{eq:phi(n)primes}) and
   $|E(\mathbb{Z}_{p^\alpha})|=(p^\alpha-1)(p^\alpha-1)/2 >p^\alpha/p$
by Theorem \ref{lem:PGcyclicprimepower1}.  Thus we may assume that
$n$ is not a prime power.

Each of the $\phi(n)$-many generators of  $\mathbb{Z}_n$ has degree $n-1$, as does the identity.  So summing over the generators and the identity gives
\[ \sum_{z} \deg(z)
         = \phi(n) (n-1) + n -1 =  (\phi(n)+1) (n-1).\]
If $\phi(n) > n/q$, then
\[ \sum_{z\in \mathbb{Z}_n} \deg(z)
         > (n-1)\left(\frac{n}{q}+1\right)>\frac{n^{2}}{q}.\]

 By (\ref{eq:degz-Zn}),  $z \in \mathbb{Z}_{n}$ has degree
        $\deg(z) =
               \ord(z) -1-\phi( \ord(z) ) + \sum_{ d | n/\ord(z)} \phi(d\ord(z))$.
For each of the ($n-\phi(n)-1$)-many nonidentity nongenerators $z$ of $\mathbb{Z}_n$,  the summand corresponding to $d=1$ is $\phi( \ord(z) )$ and to
      $d=n/\ord(z)$ is $\phi(n)$.
By (\ref{eq:phigeqn/p}), $\phi(n)\geq n/p$.  Thus
         $\deg(z)\geq \ord(z) -1 +n/p$.
Equality holds  for all nonidentity nongenerators if and only if $n=6$.
Now\footnote{The sign error originated in the following line.}
\[ \sum_{z\in \mathbb{Z}_n} \deg(z)
         \geq (\sum_{z\in \mathbb{Z}_n}  \ord(z)) - n
                       +( n-\varphi(n)-1) \frac{n}{p} + n-1
          \geq (\sum_{z\in \mathbb{Z}_n}  \ord(z))
           + \frac{n^2}{p}
                        -(\phi(n)+1)\frac{n }{p} - 1. \]
If $\phi(n) \leq n/q$,  then by (\ref{eq:betterbound})
\[  \sum_{z\in \mathbb{Z}_n} \deg(z)
     \geq     \frac{q+1}{q}\cdot \frac{n^2}{p}
                          + \frac{n^2}{p} -(\frac{n}{q}+1)\frac{n }{p} - 1
      = 2 \frac{n^2}{p} -\frac{n}{p}-1.\]
\end{proof}

Recall that for any undirected graph $\Gamma$
\[  |E(\Gamma)|  = \frac{1}{2}\sum_{g\in \Gamma} \deg(g).\]

\begin{lem}
\label{lem:phi(n)ineq}
With Notation \ref{nta:nfactored}, suppose
$\phi(n) > n/q$. Then $n$ is odd.

\end{lem}

\begin{proof} In light of (\ref{eq:phi(n)primes}),
    $\phi(n) > n/q$
if and only if
    $(p_{1} - 1) \cdots (p_{k} - 1) > p_{2} \cdots p_{k}$.
Assume $q=p_1 = 2$, so $p_{1} - 1 = 1$.  Then
the above inequality gives
    $(p_{2} - 1) \cdots (p_{k} - 1) > p_{2}\cdots p_{k}$,
which is impossible. Therefore $n$ is odd.
\end{proof}

\begin{lem}
\label{lem:d(r)}
Assume that $r$ and $s$ are natural numbers such
that  $s | r$. Then $2r - \phi(r) - 1 \geq 2s - \phi(s) - 1$, with equality if and only if $s=r$.
\end{lem}

\begin{proof}
Write  $r= st$, and observe that that $\phi(st) \leq \phi(s)t$.  Then
\[
2 r - \phi(r) - 1
     \geq 2st - \phi(s)t -1
        = t(2s - \phi(s)) - 1
         \geq 2s - \phi(s) - 1.
\]
Observe that if $t>1$, then this inequality is strict.
\end{proof}

\section{Direct products}

Let $G$ be a finite group.  Suppose $G=U\times V$ is the direct product of normal subgroups $U$ and $V$.  We view the underlying set of $G$ as the cartesian product  $U\times V =\{uv\,|\, u\in U, v\in V\}$, where we use juxtaposition (and suggestive notation) to avoid confusion with directed edges in a power graph.  We use $\cdot$ for the usual the group product on $U\times V$, namely $uv\cdot u'v' = (uu')(vv')$.
We write $g\rightarrow h$
when
    $(g,h)\in \overrightarrow{E}(G)$,
that is, when $h=g^a$ but $h\not=g$ for some $a$.

\begin{lem}
\label{lem:gendirectprod}
Let $G$ be a finite group, and let $U$ and $V$ be subgroups of $G$.
Define
\begin{eqnarray*}
 D &=&
\begin{array}{l}
 \{(uv, u'v')\,|\, uu'\in \overrightarrow{E}(U), vv'\in \overrightarrow{E}(V)\}\\
 \phantom{X}\cup
     \{(uv, u'v)\,|\, uu'\in \overrightarrow{E}(U), v\in V\}
        \cup
     \{(uv, uv')\,|\, u\in U,  vv'\in \overrightarrow{E}(V)\},
\end{array}\\
B &=& \begin{array}{l}
 \{\{uv, u'v'\}\,|\, uu' \in \overleftrightarrow{E}(U), vv' \in \overleftrightarrow{E}(V)\}\\
\phantom{X}\cup
     \{\{uv, u'v\}\,|\, uu' \in \overleftrightarrow{E}(U), v\in V\}
       \cup
     \{\{uv, uv'\}\,|\, u\in U,  vv' \in \overleftrightarrow{E}(V)\}.
\end{array}
\end{eqnarray*}
Suppose $G$ is the direct product $U\times V$ of $U$ and $V$. Then
\begin{eqnarray}
\label{eq:dEuvin}
\overrightarrow{E}(U\times V) &\subseteq& D,
\\
\label{eq:bdEuvin}
 \overleftrightarrow{E}(U\times V) &\subseteq& B.
\end{eqnarray}
\end{lem}

\begin{proof}
Let $u$, $u'\in U$ and $v$, $v'\in V$.
First suppose
      $uv\rightarrow u'v'$ in $\overrightarrow{\pg}(U\times V)$, so
for some positive integer $c$,
      $u'v'=(uv)^{\cdot c} = u^c v^c$.
The product is direct, so
       $u'= u^c$ and $v'=v^c$.
Thus one of the following holds:
      (1d)    $u\rightarrow u'$ in $\overrightarrow{\pg}(U)$ and $v\rightarrow v'$ in $\overrightarrow{\pg}(V)$,
      (2d)   $u'=u$ and $v\rightarrow v'$ in $\overrightarrow{\pg}(V)$,
      (3d)  $u\rightarrow u'$ in $\overrightarrow{\pg}(U)$ and $v'=v$, or
      (4d)  $u'=u$ and  $v'=v$.
Thus (\ref{eq:dEuvin}) holds.
Now suppose
      $u'v'\leftrightarrow uv$ in $\overleftrightarrow{\pg}(U\times V)$.
As above, one of the following holds:
      (1u)    $u'\leftrightarrow u$ in $\overleftrightarrow{\pg}(U)$ and $v'\leftrightarrow v$ in $\overleftrightarrow{\pg}(V)$,
      (2u)   $u'=u$ and $v'\leftrightarrow v$ in $\overleftrightarrow{\pg}(V)$,
      (3u)  $u'\leftrightarrow u$ in $\overleftrightarrow{\pg}(U)$ and $v'=v$, or
      (4u)  $u'=u$ and  $v'=v$.
Thus (\ref{eq:bdEuvin}) holds.
\end{proof}

\begin{lem}
\label{lem:directprodeq} With reference to Lemma
\ref{lem:gendirectprod}, suppose $G=U\times V$.  Then the
following hold.
\begin{enumerate}

\item $\overrightarrow{E}(U\times V) = D$ if and only if $(|U|,|V|) = 1$.
\item If $(|U|,|V|) = 1$, then $\overleftrightarrow{E}(U\times V) =
B$.

\end{enumerate}
\end{lem}

\begin{proof}(i): Suppose that $\gcd(|U|, |V|) = 1$. Let $u \rightarrow u'$ in
$\overrightarrow{\pg}(U)$ and $v \rightarrow v'$ in
$\overrightarrow{\pg}(V)$. Then $uv \rightarrow u'v'$ in
$\overrightarrow{\pg}(U\times V)$. Also for each $u \in U$ and $v
\in V$ we have $uv \rightarrow uv'$ in
$\overrightarrow{\pg}(U\times V)$ and $uv \rightarrow u'v$ in
$\overrightarrow{\pg}(U\times V)$. Therefore $D \subseteq
\overrightarrow{E}(U\times V)$, and thus equality holds in this
case.

We now show that there are no other instances of equality. Suppose
some prime $p$ divides $\gcd(|U|, |V|)$.  Then $U$ and $V$ contain
respective elements $u$ and $v$ of order $p$. Observe that there
is no edge $uv\rightarrow ue$ although it appears in the third subset
in the definition of $D$. Thus the inclusion is strict unless no
prime divides  $\gcd(|U|, |V|)$, i.e., $\gcd(|U|, |V|)=1$.

(ii): The proof is similar to the first part of the
proof of (i).
\end{proof}

\begin{cor}
\label{cor:directprodsize} With reference to Lemma
\ref{lem:gendirectprod}, suppose that $G = U\times V$ and
$(|U|,|V|) = 1$. Then
\begin{enumerate}
\item The edge set of the directed power graph of $U\times V$  has size
\[
|\overrightarrow{E}(U\times V)| =
         |\overrightarrow{E}(U)|\cdot |\overrightarrow{E}(V)|
         + |\overrightarrow{E}(U)|\cdot |V|
       + |U| \cdot |\overrightarrow{E}(V)|.
\]

\item The set of bidirectional edges of the directed power graph of $U\times V$ has size
\[
 |\overleftrightarrow{E}(U \times V)| =
 2|\overleftrightarrow{E}(U)|\cdot |\overleftrightarrow{E}(V)| + |\overleftrightarrow{E}(U)|\cdot |V| + |U|\cdot |\overleftrightarrow{E}(V)|. \]

\end{enumerate}
\end{cor}

\begin{proof}
Referring to the proof of (\ref{eq:dEuvin}), there are
       $|\overrightarrow{E}(U)|\cdot   |\overrightarrow{E}(V)|$
pairs in case (1d),
        $|U|   \cdot   |\overrightarrow{E}(V)|$
pairs in case (2d),
       $|\overrightarrow{E}(U)|\cdot   |V|$
pairs in case (3d), but no possible edges in case (4d). Therefore
(i) holds.  Referring to the proof of (\ref{eq:bdEuvin}), there are
       $2|\overleftrightarrow{E}(U)|\cdot   |\overleftrightarrow{E}(V)|$
pairs in case (1u) since for all $\{u,u'\} \in
\overleftrightarrow{\pg}(U)$ and $\{v,v'\} \in
\overleftrightarrow{\pg}(V)$ we have that $u'v'\leftrightarrow
uv$ and $u'v\leftrightarrow uv'$ are distinct bidirectional edges
in $\overleftrightarrow{E}(G)$.
There are
        $|U|   \cdot   |\overleftrightarrow{E}(V)|$
pairs in case (2u), and
       $|\overleftrightarrow{E}(U)|\cdot   |V|$
pairs in case (3u), but no possible edges in case (4u). Thus (ii) holds.
\end{proof}

\begin{lem}
\label{lem:prodedgedifference} Let $U$, $V$, and $V'$ be finite
groups with $|V|=|V'|$ and $(|U|,|V|) = 1$. Then
\[
 |{E}(U\times V)| - |{E}(U\times V')|
 = (|\overrightarrow{E}(U)| -2|\overleftrightarrow{E}(U)| )
           ( |\overrightarrow{E}(V)|- |\overrightarrow{E}(V')|)
        + (2|\overleftrightarrow{E}(U)| + |U|) (|{E}(V)| - |{E}(V')|).
 \]
\end{lem}

\begin{proof}
Expand   $|{E}(U\times V)| - |{E}(U\times V')|$ with (\ref{eq:EGsize}):
\[
|{E}(U\times V)| - |{E}(U\times V')|=( |\overrightarrow{E}(U\times V)|-|\overrightarrow{E}(U\times V)| )
    - (|\overleftrightarrow{E}(U\times V)| - |\overleftrightarrow{E}(U\times V')|).
\]
Since $|V|=|V'|$, Corollary \ref{cor:directprodsize} gives
 \begin{eqnarray*}
\overrightarrow{E}(U\times V)| -|\overrightarrow{E}(U\times V')|
   &=& (|\overrightarrow{E}(U)| +|U|)
           ( |\overrightarrow{E}(V)|- |\overrightarrow{E}(V')|),\\
|\overleftrightarrow{E}(U\times V)| - |\overleftrightarrow{E}(U\times V')|
     &=& (2|\overleftrightarrow{E}(U)| + |U|)
            (|\overleftrightarrow{E}(V)|-|\overleftrightarrow{E}(V')|).
\end{eqnarray*}
In the second line, use (\ref{eq:EGsize}) to write $|\overleftrightarrow{E}(V)| = |\overrightarrow{E}(V)|-|{E}(V)|$, and similarly for  $|\overleftrightarrow{E}(V')|$. Thus
\[
|\overleftrightarrow{E}(U\times V)| - |\overleftrightarrow{E}(U\times V')|
 = (2|\overleftrightarrow{E}(U)| + |U|)
            ((|\overrightarrow{E}(V)|-|\overrightarrow{E}(V')|) - (|E(V)|-|E(V')|)).
\]
Combining the above and simplifying gives the desired result.
\end{proof}

\begin{cor}
\label{cor:whenprodhasmoreedges}
With reference to Lemma \ref{lem:prodedgedifference},
suppose $|E(V)|\geq |E(V')|$ and
    $|\overrightarrow{E}(V)|\geq |\overrightarrow{E}(V')|$.
Then      $|{E}(U\times V)| \geq |{E}(U\times V')|$, with equality if
and only if $|E(V)|= |E(V')|$ and
    $|\overrightarrow{E}(V)|= |\overrightarrow{E}(V')|$.
\end{cor}

\begin{proof}
Observe that
    $|\overrightarrow{E}(U)| > 2|\overleftrightarrow{E}(U)|$
since each bidirectional edge is counted twice in $\overrightarrow{E}(U)$ and since $\overrightarrow{E}(U)$ contains one-way edges from each element of $U$ to the identity.   Clearly,
    $2|\overleftrightarrow{E}(U)| + |U|>0$.
With the given assumptions,
Lemma \ref{lem:prodedgedifference} gives that
    $|{E}(U\times V)| - |{E}(U\times V')|\geq 0$,
so the result follows. Equality holds if and only if
   $|E(V)|= |E(V')|$ and
   $|\overrightarrow{E}(V)|= |\overrightarrow{E}(V')|$ since all
terms on the right are positive, so the right side is zero if and
only if these equalities hold.
\end{proof}

\section{Semidirect products}

We recall  semidirect products.  Recall that a group $G$ is the (internal) semidirect product of a normal subgroup $U$ and a subgroup $V$ if and only if  $G=UV$ and $U\cap V=\{e\}$ \cite{r}.
To uniquely  determine $G$ from $U$ and $V$, we specify a homomorphism $\varphi:V\rightarrow \mathrm{Aut}(U)$.  As is the custom, write  $G=U\rtimes_\varphi V$ in this situation.
When there is no ambiguity, we write $\varphi v$ for $\varphi(v)$ and place such automorphisms as a superscript of the element of $U$ to which it is applied.

The elements of $G=U\rtimes_\varphi V$ can be identified with the cartesian product of the underlying sets of $U$ and $V$, just as is the case for the direct product.  We use $\ast$ for the group product on $U\rtimes_\varphi V$, namely $uv \ast u'v' = (u(u')^{\varphi v})(vv')$.
We write $(uv)^{ \cdot c}$ and $(uv)^{\ast c}$ to denote the $c$th powers of $uv$ under the corresponding operations.  We write $\ord_{\cdot}(uv)$ and $\ord_{\ast}(uv)$ to denote the order of $uv$ relative to the corresponding multiplication, and  we write $(uv)^{\cdot{-}1}$ and $(uv)^{\ast{-}1}$ for the corresponding inverses.  With this notation,
\begin{eqnarray}
\label{eq:semidirect} uv^{\ast c} &=&
(uu^{\varphi v}  u^{\varphi v^2} \cdots  u^{\varphi v^{c-2}} u^{\varphi v^{c-1}}) (v^c), \\
(uv)^{\ast-1} &=& (u^{-1})^{\varphi v^{-1}}(v^{-1}) \nonumber.
\end{eqnarray}

\begin{lem}
\label{lem:astpowercyclic} Suppose that $G$ is a finite group and
that $G=U\rtimes_\varphi V$ is the  semidirect product of a normal
cyclic subgroup $U$ and a subgroup $V$.  Pick $u\in U$ and $v\in
V$. Then  $u^{\varphi v} = u^r$ for some $r$, and
          $(uv)^{\ast c}  = u^t v^c$,
where $t = 1+r+r^2 +\cdots +r^{c-1}$.
\end{lem}

\begin{proof}
Consider the subgroup $\langle u \rangle$ of $U$.  Since $U$ is
cyclic and since $\langle u \rangle$ is the unique
subgroup of $U$ of its order, it must be the case that $\langle u
\rangle$ is a characteristic subgroup of $U$.  Thus $u^{\varphi v} = u^r$
for some $r$. Now  $u^{\varphi v^b} = u^{r^b}$.  With this we
compute
\[ (u  v)^{ \ast c}
   = (uu^{\varphi v}  u^{\varphi v^2} \cdots  u^{\varphi v^{c-2}} u^{\varphi v^{c-1}}) (v^c)
    = (u^1 u^r u^{r^2} \cdots u^{r^{c-1}})  (v^c)
    = (u^{1+r+r^2 +\cdots +r^{c-1}})  (v^c) = u^t v^c.
\]
\end{proof}

\begin{lem}
\label{lem:Semidirprodcyc}
Suppose that $G$ is a finite group and that $G=U\rtimes_\varphi V$ is the  semidirect product of a normal cyclic subgroup $U$ and a subgroup $V$.
Then
\begin{eqnarray}
\label{eq:sdEuvin} \overrightarrow{E}(U\rtimes_\varphi V) &\subseteq& D,
\\
\label{eq:sbdEuvin}
 \overleftrightarrow{E}(U\rtimes_\varphi V) &\subseteq& B.
\end{eqnarray}
In particular, if   $(|U|,|V|)=1$, then $E(U\rtimes_\varphi
V)\subseteq E(U\times V)$ and
$\overleftrightarrow{E}(\mathbb{Z}_{|U|}\rtimes_\varphi
\mathbb{Z}_{|V|})
           \subseteq
     \overleftrightarrow{E}(\mathbb{Z}_{|U|}\times \mathbb{Z}_{|V|})$.
\end{lem}

\begin{proof}
Pick $u$, $u'\in U$ and $v$, $v'\in V$, and say $u^{\varphi v} =
u^r$. Suppose  $uv\rightarrow u'v'$ in
$\overrightarrow{\pg}(U\rtimes_\varphi V)$. Then
     $u'v' = (uv)^{\ast c} = u^t v^c$
for some $t$ as in Lemma \ref{lem:astpowercyclic}.  Since the
product is semidirect,
        $u'=u^{t}$ and $v'=v^c$.
Arguing as in the proof of Lemma \ref{lem:gendirectprod}, we reach the same conclusion.
\end{proof}

\begin{lem}
\label{lem:divideorder1}
Suppose that $G$ is a finite group and
that $G=U\rtimes_\varphi V$ is the  semidirect product of a normal
abelian subgroup $U$ and a subgroup $V$.   Assume $U$ and  $V$
have coprime orders. Then  for all $u \in U$ and $v\in V$,
     $\ord_{\ast}(uv)|\ord_{\cdot}(uv)$.
\end{lem}

\begin{proof}
To prove the result  we show that $(u v)^{\ast
\ord_{\cdot }({uv})}$ is the identity of $U\rtimes_\varphi V$. Note that
   $\ord_{\cdot}(u v) = \ord(u)\ord(v)/\gcd(\ord(u),\ord(v))
                              = \ord(u)\ord(v)$.
By (\ref{eq:semidirect}),
\begin{eqnarray*}
      (u  v)^{ \ast \ord(v)}  &=&
(uu^{\varphi v}  u^{\varphi v^2} \cdots  u^{\varphi v^{\ord(v)-2}}
u^{\varphi v^{\ord(v)-1}}) (v^{\ord(v)})\\ &=& (uu^{\varphi v}
u^{\varphi v^2} \cdots  u^{\varphi v^{\ord(v)-2}} u^{\varphi
v^{\ord(v)-1}})(e).
\end{eqnarray*}

Thus
\begin{eqnarray*}
(u    v)^{ \ast \ord_{\cdot}({u v})}  &=&\left[ (uu^{\varphi v}
u^{\varphi v^2} \cdots  u^{\varphi v^{\ord(v)-2}} u^{\varphi
v^{\ord(v)-1}})(e)\right]^{\ast \ord(u)}\\
    &=& (u^{\ord(u)}(u^{\varphi v})^{\ord(u)}
(u^{\varphi v^2})^{\ord(u)} \cdots  (u^{\varphi
v^{\ord(v)-2}})^{\ord(u)} (u^{\varphi
v^{\ord(v)-1}})^{\ord(u)})(e)\\
    &=& (e) (e),
    \end{eqnarray*} since $U$ is abelian, as required.
\end{proof}

For later use we recall a couple facts concerning involving semidirect products.

\begin{lem}\cite[Theorem 10.30]{r} (The Schur--Zassenhaus theorem)
\label{lem:schurzass} Let $G$ be a finite group, and let $K$ be a normal subgroup of  $G$ with $(|K|, |G:K|) = 1$. Then $G$ is a semidirect product of $K$ and $G/K$.  In particular, there exists a subgroup $H$ of $G$ with order $|G : K|$ satisfying $G = HK$ and $H\cap K=\{e\}$.
\end{lem}

\begin{lem}\cite[Theorem 1.2 (i)]{b1}
\label{lem:normalindex}
Let $p$ be an odd prime, and suppose that $P$ is a non-abelian
$p$-group with a cyclic subgroup of index $p$. Then $P \cong
\mathbb{Z}_{p^{\alpha-1}}\rtimes_\varphi \mathbb{Z}_{p}$, the center of $P$ has order $|Z(P)| = p^{\alpha - 2}$, and $P$ has presentation
\[ P \cong M_{p^{\alpha}} = \langle a,b | a^{p^{\alpha - 1}} = b^{p} = 1 ,
b^{-1}ab = a^{1+p^{\alpha - 2}} \rangle.
\]
\end{lem}

\section{Some group theory}

We need   a few more results from group theory.

\begin{defn}
\label{def:pcomp}
Let $p$ be a prime. Let $G$ be a finite group, and let $P$ be a Sylow
$p$-subgroup of $G$. A {\em $p$-complement} in $G$ is a  subgroup with index equal to the order of a Sylow $p$-subgroup.
\end{defn}

\begin{thm} \cite[Theorem 10.21]{r}  (Burnside's transfer theorem)
With the notation of Definition \ref{def:pcomp},
if $P\subseteq Z(N_G(P))$, then $G$ has a normal $p$-complement.
\end{thm}

\begin{cor}
\cite[Corollary 10.24]{r}
\label{cor:qcomplementsmall}
With the notation of Definition \ref{def:pcomp},
suppose $G$ contains a cyclic $q$-Sylow subgroup, where $q$ is the least prime divisor of $|G|$. Then $G$ has a normal $q$-complement.
\end{cor}

\begin{lem}
\label{lem:normalS}
A finite group containing a cyclic subgroup of prime index  has a non-trivial normal Sylow subgroup.
\end{lem}

\begin{proof}
Let $G$ be a finite group and suppose $C$ is a cyclic subgroup of prime index $p$ in $G$.  Induct on the number of distinct prime factors of $|G|$.
If $G$ is a $p$-group, then $G$ itself is a normal Sylow $p$-subgroup.
Assume that there is a prime $r$ different from $p$ which divides $|G|$.

 Let $R$ be a Sylow $r$-subgroup of $C$.
 Now
        $|G:R| = p\cdot |C:R|$
is not divisible by $r$, so $R$ is a Sylow $r$-subgroup of $G$.  We are done if $R$ is normal in $G$,  so assume that this is not the case.
Now $R$ is normal in the cyclic subgroup $C$, so $C\leq N_G(R)< G$.
Since $C$ has prime index in $G$, it must be the case that $N_G(R)=C$.
In particular, $R\leq Z(N_G(R))$, so $G$ has a normal $r$-complement $N$ by Burnside's transfer theorem.    Thus $RN=G$, and hence $CN=G$ as well.

Now  $N\cap C$ is cyclic since $C$ is, and
        $|N:N\cap C|=|CN:C|=|G:C| = p$.
  Suppose $N\cap C$ is nontrivial. Then by induction $N$ has  a nontrivial normal Sylow subgroup $S$.  In fact, $S$ is characteristic in $N$, so $S$ is normal in $G$. Observe that $S$ is a Sylow subgroup of $G$ since $|G:N|$ and $|N|$ are coprime.

Suppose
    $N \cap C = 1$.
Then $|G|=|NC|=|N||C|/|N\cap C| = (|G|/|R|)(|G|/p)/1$, so $|G|=p
r^\alpha$ for some positive power $\alpha$. Now  $R$ and $P$ are cyclic subgroups of $G$, so Corollary \ref{cor:qcomplementsmall}  implies that if $r < p$ then $G$ has a normal $r$-complement, and if $p < r$ then $G$ has a normal $p$-complement.
In either case, there is a normal Sylow subgroup of $G$.
\end{proof}

\begin{lem}
\label{lem:divideorder2}
Let $G$ be a finite group of order $n$.
Suppose $n$ is not a power of  2, and
  let $r>2$ be a prime divisor of $n$.
Suppose that $G$ contains both
  a cyclic $r$-complement and
  a normal $r$-Sylow subgroup which itself contains
        a cyclic subgroup of index $r$.
Then there is a bijection $\lambda : G \rightarrow \mathbb{Z}_{n}$
such that
        $\ord(g)|\ord(\lambda(g))$ for each $g \in G$.
\end{lem}

\begin{proof}
Let $U$ be a cyclic $r$-complement of $G$.  Let $R$ denote a normal Sylow $r$-subgroup of $G$ which contains a cyclic subgroup $C$ with index $r$.  Say $|R|=r^\alpha$.  By consideration of group orders,  $G=RU$ and $R\cap U=\{e\}$.  Since $R$ is normal in $G$, both  $G = R\rtimes_\varphi U$ and  $G/R\cong U$.

Suppose $R$ is nonabelian. By Lemma \ref{lem:normalindex}, the center
$Z(R)$ of $R$ has order $r^{\alpha-2}$.  Now $Z(R)\subseteq C$
since otherwise by consideration of indices $R=Z(R)C$; this would
imply that $R$ is abelian, contrary to our assumption.  Hence
$Z(R)$ is cyclic; moreover, it is the unique subgroup of $C$ with
index $r$. Now $Z(R)$ is a characteristic subgroup of the normal
(characteristic) subgroup $R$, so $Z(R)$ is normal in $G$.   Since
$Z(R)$ is a normal cyclic subgroup of $G$,  Lemma
\ref{lem:divideorder1} gives that
      $\ord_{\ast}(ab) | \ord_{\cdot}(ab)$
      for all
      $ab\in Z(R)\rtimes_\varphi U$.
Observe that $Z(R)\times U$ is cyclic.   Thus we may take $\lambda$  restricted to $Z(R)\rtimes_\varphi U$ to be a bijection from
   $Z(R)\rtimes_\varphi U$
   to
   $Z(R)\times U \cong\mathbb{Z}_{|Z(R)|}\times \mathbb{Z}_{|U|}
                      \subseteq \mathbb{Z}_{n}$
for which $\ord(x)|\ord(\lambda(x))$ for each $x\in Z(R)\rtimes_\varphi U$.

Now since $G = RU$, we can write $G = Z(R)U \cup (R - Z(R))U$.
Suppose that $x \in (R - Z(R))U$. Then $x = ru$ where $r \in R -
Z(R)$ and $u \in U$. Since $R$ is not cyclic it contains no
element of order $r^\alpha$. Hence the order of $x$ is divisible
by $r^{\alpha - 1}\ord(u)$. We identify $\mathbb{Z}_{|U|}$
with $U$, so
 $\mathbb{Z}_n
   = \mathbb{Z}_{|Z(R)|}U \cup (\mathbb{Z}_{r^{\alpha}}
                                                      - \mathbb{Z}_{|Z(R)|})U$.
Since the elements of $\mathbb{Z}_{r^{\alpha}} - \mathbb{Z}_{|Z(R)|}$ have order $r^{\alpha - 1}$ or $r^{\alpha}$ it follows that
   $r^{\alpha -1}\ord(u) | \ord(z)\ord(u)$ for each
 $z \in \mathbb{Z}_{r^{\alpha}} - \mathbb{Z}_{|Z(R)|}$.
Thus any bijection from
   $G-(Z(R)\rtimes_\varphi U)$ to
   $\mathbb{Z}_{n}-\mathbb{Z}_{|Z(R)|}\mathbb{Z}_{|U|}$
will extend $\lambda$ to all of $G$ with the desired properties.

Now suppose $R$ is abelian. If $R$ is cyclic then by Lemma
\ref{lem:divideorder1} we have nothing to prove. Suppose $R$ is
not cyclic. Then Lemma \ref{lem:divideorder1} gives that
$\ord_{\ast}(ab) | \ord_{\cdot}(ab)$ for all $ab\in
R\rtimes_\varphi U$. Now it is enough to show that there is a
bijection from $R\times U$ to $\mathbb{Z}_n$ with desired
property. By the assumption $R$ has a cyclic subgroup of index
$r$. Thus $R \cong \mathbb{Z}_{r^{\alpha - 1}} \times
\mathbb{Z}_r$, and $R\times U \cong \mathbb{Z}_{r^{\alpha - 1}}
\times \mathbb{Z}_r \times U$. Now if we write $R\times U = (R -
\mathbb{Z}_{r^{\alpha - 1}})U \cup \mathbb{Z}_{r^{\alpha - 1}}U$
then arguing as in the above we reach the same conclusion.
\end{proof}

\begin{lem}
\label{lem:orddivbij:E=E,cyc}
Let $G$ be a finite group, and suppose that  there is a bijection
     $\lambda : G \rightarrow \mathbb{Z}_{n}$
such that
        $\ord(g)|\ord(\lambda(g))$ for each $g \in G$.
Then
$|E(G)| =|E(\mathbb{Z}_n)|$ if and only if $G\cong  \mathbb{Z}_n$.
\end{lem}

\begin{proof}
Suppose $|E(G)| =|E(\mathbb{Z}_n)|$.
By Lemma \ref{lem:d(r)}, for all $g\in G$,
    $2\ord(g) -\phi(\ord(g)) - 1
           \leq 2\ord(\lambda(g)) - \phi(\ord(\lambda(g))) -1$.
By  (\ref{eq:EGsize}), and since $\lambda$ is a bijection 2$|E(G)|
= \sum_{g \in G}2\ord(g) - \phi(\ord(g)) - 1
           = \sum_{g \in G}2\ord(\lambda(g)) - \phi(\ord(\lambda(g))) - 1
           = 2|E(\mathbb{Z}_{n})|$.
This equality and the preceding inequality imply that for all
$g\in G$, $2\ord(g) - \phi(\ord(g)) = 2\ord(\lambda(g)) -
\phi(\ord(\lambda(g)))$.

Pick a generator $z$ of $\mathbb{Z}_n$, and let $g=
\lambda^{-1}(z)$.  Then
     $2\ord(g) -\phi(\ord(g))
            = 2\ord(\lambda(g)) - \phi(\ord(\lambda(g)))
            = 2\ord(z) - \phi(\ord(z))
            =2n - \phi(n)$.
Suppose $G$ is not cyclic.  Then  $\ord(g)<n$  and $\ord(g)$ divides $n=\ord(z)$.  Lemma \ref{lem:d(r)} implies that   $2\ord(g) -\phi(\ord(g)) <2n - \phi(n)$.  This contradicts the above.  Thus $G$ must be cyclic, and hence isomorphic to  $\mathbb{Z}_n$. The converse is clear.
\end{proof}

\section{Proof of main theorem}
The proof of Theorem \ref{thm:main} is developed in a series of technical lemmas.

\begin{nta}
\label{nta:groupnwfactors}
Let $G$ be a finite group of order $n$, and adopt the conventions of Notation \ref{nta:nfactored} for the prime factorization of $n$.
\end{nta}

\begin{lem}
\label{lem:cyc,primnotce}
With Notation \ref{nta:groupnwfactors}, the following hold.
\begin{enumerate}
\item No cyclic group is a counterexample to Theorem \ref{thm:main}.

\item No group of prime power order is a counterexample to Theorem \ref{thm:main}.
\end{enumerate}
\end{lem}

\begin{proof}
Part (i) is clear, and Part (ii) follows from Theorem
\ref{lem:PGcyclicprimepower1}.
\end{proof}

\begin{lem}
\label{lem:cynormsylnotce}
With Notation \ref{nta:groupnwfactors},  suppose $G=P\rtimes_\phi T$ is the semidirect product of a normal cyclic Sylow subgroup $P$ and a subgroup $T$ with order coprime to that of $P$.
Then $G$ is not a minimal counterexample to Theorem \ref{thm:main}.
\end{lem}

\begin{proof}
Suppose $G$ is  a minimal counterexample to Theorem \ref{thm:main}.
By Lemma
\ref{lem:Semidirprodcyc},
    $|E(P\times T)|\geq |E(P\rtimes_\varphi T)|
                             =    |E(G)|
                           \geq |E(\mathbb{Z}_{n})|$.

Note that $P$ is isomorphic to $\mathbb{Z}_{|P|}$. By construction  $\gcd(|P|,|T|)=1$. For the sake of comparison, let $T' = \mathbb{Z}_{|T|}$, and observe that $P\times T'$ is isomorphic to  $\mathbb{Z}_{n}$.  Identify $\mathbb{Z}_{n}$ and $P \times  T'$.

Suppose $T$ is not cyclic.   Since $|T|<|G|$, and since $G$ is
assumed to be a minimal counterexample,  $|E(T)|< |E(T')|$.  Also,
by Theorem \ref{cor:maximume}, $|\overrightarrow{E}(T)| <
|\overrightarrow{E}(T')|$. Now  by Corollary
\ref{cor:whenprodhasmoreedges},
      $|E(P\times T)| < |E(P\times T')|= |E(\mathbb{Z}_{n})|$.
This implies that $|E(\mathbb{Z}_{n})|>|E(\mathbb{Z}_{n})|$, which is absurd.  This contradiction leads to the conclusion that in this case, $G$ is not a minimal counterexample to Theorem \ref{thm:main}.

Suppose $T$ is cyclic.  Then $G$ is the semidirect product of
cyclic subgroups of coprime orders. Note that $G$ is not cyclic by
Lemma \ref{lem:cyc,primnotce}. Now is $P\times T$; in fact, $P\times T$ is isomorphic to $\mathbb{Z}_n$ since $\gcd(|P|,|T|)=1$.   Thus
      $|E(G)| \leq |E(\mathbb{Z}_{n})|= |E(P\times T)|$.
It remains to show that equality does not hold.

Suppose   $|E(G)| = |E(\mathbb{Z}_{n})|$.  Since $T$ is  a cyclic
$p$-complement and $P$ is a cyclic Sylow $p$-subgroup, Lemma
\ref{lem:divideorder1} gives that there is a bijection
$\theta:G\rightarrow\mathbb{Z}_{n}$ such that
      $\ord(g) \ | \ \ord(\theta(g))$ for all $g \in G$.
Now Lemma \ref{lem:orddivbij:E=E,cyc}, leads to the conclusion that $G$ is cyclic.    This contradiction implies that in this case, $G$ is not a minimal counterexample to Theorem \ref{thm:main} either.  The result follows.
\end{proof}

\begin{lem}
\label{lem:largeordntce}
With Notation \ref{nta:groupnwfactors}, if there exists  $g\in G$ with
$\ord(g) > n / p$, then $G$ is not a minimal counterexample to Theorem \ref{thm:main}.
\end{lem}

\begin{proof}
Since $\ord(g) > n / p$, we have
       $|G:\langle g\rangle| =|G|/\ord(g) < p$.
Thus $\langle g \rangle$ contains a Sylow $p$-subgroup $P$ of
$G$. In particular, $P$ is cyclic and normalized by $\langle
g\rangle$; hence, $P \ \lhd \ G$. By Lemma
\ref{lem:schurzass}, $G = P\rtimes_\varphi T$ (semidirect
product).  The result follows from Lemma \ref{lem:cynormsylnotce}.
\end{proof}

\begin{lem}
\label{lem:lemma1}
With Notation \ref{nta:groupnwfactors}, if $\phi(n) \leq \frac{n}{q}$, then
$G$ is not a minimal counterexample to Theorem \ref{thm:main}.
\end{lem}

\begin{proof}
Suppose $G$ is a minimal counterexample to Theorem \ref{thm:main}.
Then $|E(G)|\geq |E(\mathbb{Z}_n)|$.  Note that $G$ is not cyclic and  $n$ is not a prime power by Lemma \ref{lem:cyc,primnotce}.  In particular, $n$ is not a power of 2.

 By assumption and Lemmas \ref{lem:edgesizes} and \ref{lem:avez},
 \[
\sum_{g\in G} 2\ord(g) - \phi(\ord(g))-1 = \sum_{g\in G} \deg(g) =
2|E(G)| \geq 2|E(\mathbb{Z}_n)| = \sum_{z\in \mathbb{Z}_n}
\deg(z)\geq 2 \frac{n^2}{p} - \frac{n}{p}-1.
\]
Cancel  -1 on the right and $-\phi(\ord(e))=-1$ on the left, and then drop the remaining $- \phi(\ord(g))$ on the left.  Add $n=  \sum_{g\in G} 1$ to both sides to find
\[
\sum_{g\in G} 2\ord(g) >
         2 \frac{n^2}{p} - \frac{n}{p} +n.
\]
Thus there is at least one $g\in G$ for which $\ord(g) > n / p$.
The result follows from Lemma \ref{lem:largeordntce}.
\end{proof}

\begin{lem}
\label{lem:cenotphi>n/q}
With Notation \ref{nta:groupnwfactors}, if $\phi(n) > \frac{n}{q}$, then
$G$ is not a minimal counterexample to Theorem \ref{thm:main}.
\end{lem}
\begin{proof}
Suppose  $G$ is a counterexample of minimal order. Then
arguing as in Lemma \ref{lem:lemma1} we find
\[
\sum_{g\in G} 2\ord(g) - \phi(\ord(g))-1 \geq (n-1)\left(\frac{n}{q}+1\right).
\]
Each term $\phi(\ord(g))$ is at least one, so we may add $2n$ to
both sides and then drop all remaining contributions of
$-\phi(\ord(g))$.  Also subtract $2$, the contribution of $g=e$,
from each side:
\[
\sum_{g\in G-\{e\}} 2\ord(g)  \geq (n-1)\left(\frac{n}{q}+1\right)+ 2n-2.
\]
Thus among the $n-1$ terms on the left, there is at least one nonidentity $g$ with $\ord(g)\geq n/2q + 1/2 + 1$. In particular, $\ord(g)> n/2q$.

Suppose $p > 2q$. Then  $n / p < n / 2q$, and  thus
    $  \ord(g) > n/2q > n/p$.
In this case the result follows from Lemma \ref{lem:largeordntce}.
Note $p \not= 2q$ since $p$ is prime.

Now suppose  $p < 2q$  (this case occurs when $n$ is the product
of twin primes, for example). Then  $\ord(g) > n / 2q$, so
     $|G : \langle g \rangle| = |G|/\ord(g) < n / (n/2q) = 2q$.
Note that  $n$ is odd by Lemma \ref{lem:phi(n)ineq}, and thus
    $|G : \langle g \rangle|$ is a prime number, say $s$.
Now by Lemma \ref{lem:normalS}, $G$ has a normal $r$-Sylow subgroup
$R$ for some prime $r$.  By Lemma
\ref{lem:schurzass},  $G$ contains an $r$-complement $U$. Thus
$G=R\rtimes_\varphi U$.

 Note  that $\langle g \rangle$ contains a Sylow $t$-subgroup of $G$ for $t\not=s$.   If $r \neq s$, then
 $R\subseteq \langle g \rangle$, so  it is cyclic. The hypotheses of Lemma \ref{lem:cynormsylnotce} are satisfied by $R$ and $U$, so we conclude that $G$ is not a minimal counterexample in this case.

 Suppose that $r = s$.  Say $|R|=r^\beta$. Since $|G:\langle g\rangle|=r$, the subgroup $H=\langle g^{r^{\beta-1}}\rangle$ has order $|H|=|G|/r^\beta$.  Thus $H$ is a cyclic $r$-complement.  Note that $q>2$ by Lemma
 \ref{lem:phi(n)ineq}.

Now Lemma \ref{lem:divideorder2} gives a bijection, say again
$\lambda$, from $G$ to $\mathbb{Z}_{n}$ such that $\ord(g) \ | \ \ord
(\lambda(g))$ for all $g\in G$. In this case Lemma
\ref{lem:orddivbij:E=E,cyc} leads to the conclusion that $G$ is
cyclic.    This contradiction implies that $G$ is not a minimal
counterexample to Theorem \ref{thm:main}.
\end{proof}

We are now ready to prove Theorem \ref{thm:main}.

\begin{proof}
Since Lemmas \ref{lem:lemma1} and \ref{lem:cenotphi>n/q} exhaust the possibilities, there is no  counterexample of minimal order to Theorem \ref{thm:main}, and hence no counterexamples at all.  Thus Theorem \ref{thm:main} holds for all groups of all orders.
\end{proof}

We may restate Theorem \ref{thm:main} without reference to power graphs as follows.

\begin{cor}
Let $G$ be a finite group of order $n$. Then
$\displaystyle{
\sum_{d | n} (2d - \phi(d))\phi(d) \geq \sum_{g\in G} (2\ord(g) -
\phi(\ord(g)))}$,
with equality if and only if $G$ is cyclic.
\end{cor}

\noindent{\bf Acknowledgment.} The second author thanks his advisors Professor Hassan Yousefi Azari and Professor Ali Reza Ashrafi for providing him with this question as a conjecture. We thank Professor I. Martin Isaacs, from the University of Wisconsin-Madison for his valuable comments. We thank Professor Andrei Kelarev from the University
of Newcastle (Australia) for  discussion on the background and uses of 
power graphs of groups.

\end{document}